\documentclass[a4paper,12pt]{article}

\usepackage[utf8]{inputenc}
\usepackage[english]{babel}
\usepackage{amssymb,amsmath,amsthm,tikz,tikz-cd}
\usepackage{graphicx}

\newcommand{\RR}{\mathbb{R}}

\newcommand{\NN}{\mathbb{N}}

\newcommand{\QQ}{\mathbb{Q}}

\newcommand{\ZZ}{\mathbb{Z}}

\newcommand{\vb}{\vec{b}}

\newcommand{\vx}{\vec{x}}
\newcommand{\sgn}{\textnormal{sgn}}
\newcommand{\Det}{\textnormal{det}}

\newtheorem{theorem}{Theorem}[section]
\newtheorem{definition}[theorem]{Definition}
\newtheorem{prop}[theorem]{Proposition}
\newtheorem{corollary}[theorem]{Corollary}
\newtheorem{lemma}[theorem]{Lemma}

\title{Conant's Metric Spectra Problem}
\author{Veljko Tolji\'{c}}

\begin{document}

\maketitle

\begin{abstract}
    In this paper, we try to minimize the scope of possible unique metric spectra up to equivalence. While it is well known that every spectra $S\subseteq \RR^+$ is equivalent to a spectra $T\subseteq \NN$, it has remained open if $T$ could also maintain a desirable combinatorial form. Conant questioned if $T= \{t_1,...,t_n \}_{<}$ could be taken such that $2^i-1 \leq t_i \leq 2^n -1$. In this paper, we come to two partial answers. The first is that the largest element $t_n$ can be chosen such that $t_n\leq 2^n $. Approximating a full solution, we also observe $T$ with the combinatorial form $2^i \leq t_i \leq 2^{n+1}$. Our methods are rather unique in the field as we utilize linear optimization and polygonal geometry to achieve our results. Our work aims to approach a full characterization of metric spectra, and simplify future computational endeavors in the field.
\end{abstract}

\section{Introduction}

The study of combinatorial properties of metric spaces has proven to be a unique gem in the movement of classifying ultrahomogeneous structures. This can be highlighted in the diversity of methodologies used. One of the most significant breaktrhoughs was the recognition of the \textit{4-value condition} by Delhomm\'e, Laflamme, Pouzet, and Sauer \cite{4value}. The 4-value condition, a combinatorial property of a set of positive reals $S\subseteq \RR^+$, is in one-to-one correspondence with the existence of a countable homogeneous metric space $\mathcal{M}$ with spectra equal to $S$. In his thesis, Nguyen van Th\'e analyzed classes of metric spaces from the perspective of the KPT correspondence, working in the intersection of structural Ramsey Theory and topological dynamics \cite{KPT, NVTThesis}. Nguyen van Th\'e also began a classification of classes of metric spaces in the appendix of his thesis by considering the cases when $|S| \leq 4$. The core consideration in the casework is whether or not the $4$-value condition is met, and why/why not. Shortly thereafter, Sauer was able tighten down on the analysis of metric spectra (or distance sets in his phrasing) by viewing $S\subseteq \RR^+$ inherently as a structure with a ternary relation. The relation in question, aptly being titled the \textit{metric triple} relation, can also be used to define the $4$-value condition and hence has become the underpinning of any classification theory for finite metric structures. Conant was able to expand this to \textit{distance magma's}, an algebraic interpretation of the metric triple condition \cite{ConantThesis}. Not only do the distance magma isomorphism classes completely coincide with the metric triple structure in the case of a magma implicitly is defined from an $S\subseteq \RR^+$, but Conant was also able to show that not every distance magma is of this form. This was done via computational verification, alongside Conant proving the following fact.

\begin{theorem}[Conant \cite{ConantThesis}]
    If $S= \{s_1,...,s_n\}_{<}\subseteq \RR^+ $, and $n\leq 6$, then there is a set $T = \{t_1,...,t_n\}_{<} \subseteq \NN$ such that for all $i\leq n$, $2^i -1 \leq t_i \leq 2^n -1$, and $T$ is isomorphic to $S$ as a distance magma (equivalently, as a metric spectra).
\end{theorem}

The natural followup question of Conant, and the main motivating question of our work, was whether or not this is true for all $n\in \NN$ \cite{OpenProblems}. This lead to our following partial result, and the main emphasis of this paper.
\begin{theorem}
    If $S\subseteq \RR^+$ and $|S| = n$, then $S$ is equivalent as a metric spectra to some $T= \{t_1,...,t_n\}_{<}\subseteq \NN$ where $t_n \leq 2^n$. Moreover, at the cost of increasing the upperbound, we can choose $T$ so that $2^i \leq t_i \leq 2^{n+1} $.
\end{theorem}

 A unique aspect of our proof is our use of linear optimization and polygonal geometry. Our first observation was that an isomorphism class of a metric spectra was uniquely determined by a set of linear inequalities. Each linear inequality defines a half-space, and consequently, the intersection of all such half-spaces defines a polygon. We can then find representatives of our isomorphism class by looking for vertices in our polyhedron. By appropriately scaling a vertex and proving some bounds on determinant computations, we get the desired integral representatives of a metric spectra. An immediate nicety of our approach is that our work is accessible to a broad range of mathematicians. We hope as a consequence, we can help further popularize the classification of metric spectra, and influence new developments and collaboration's in a field with many open problems. 
\\
\\
The organization of our paper is rather standard. Section 2 contains all the preliminary material needed to understand our main result. This includes a primer on finite metric geometry, and a primer on polygons. Given that significant definitions are mentioned in Section 2, we recommend that it not be skipped to ensure a smooth reading experience. Section 3 consists of three parts. First, we establish an immediate connection between metric spectra and cones, and show how this connection can be used to prove Conants' theorem that every metric specta admits a rational representative. We then introduce the cover poset on $\RR^n$, and show how it can be used to appropriately decompose vectors. In the third and final subsection of Section 3, we use the combinatorics of coverings to deduce upperbounds of determinant computations. Putting everything together, we deduce our main theorem. 

\section{Preliminaries}

\subsection{Finite Metric Geometry}

There are many ways one could classify classes of metric spaces. One that has proven to be quite fruitful is categorizing by \textit{metric spectrum} \cite{ConantThesis, KPT, NVTThesis}.

\begin{definition}
    We call a subset $S\subseteq \RR^+$ a metric spectrum. When $S$ is finite, we write $S=\{s_1,...,s_n \}_{<}$ to mean that $s_i$ are in monotonically increasing order.
\end{definition}

The reasoning for this classification is quite simple, if you want to structurally study finite metric spaces, you will need to express them (somehow) as a class of relational structures. Given a metric space $(X,d) $ where the image of $d$ is a subset of $S\subseteq \RR^+ $, we can identify $(X,d)$ as the relational structure $(X, \{R_s\}_{s\in S})$, where $R_s$ are symmetric binary relations with the property that $x R_s y \iff d(x,y) =s$. In this schema, substructures are precisely isometric subspaces of $X$. While this identification will be unimportant for our purposes, it is important to emphasize \textit{why} people choose to study spectras in the first place. While we may only be doing finite combinatorics and analysis of the real line, the end result is the classification of categories of metric spaces where metric spaces are interpreted as relational structures.

\begin{definition}
    Given $S\subseteq \RR^+$, we say $(a,b,c) \in S^3$ forms a \textit{metric triple} if:
    \begin{itemize}
        \item $c\leq a+b$
        \item $b\leq a+c$
        \item $a \leq b+c$
    \end{itemize}
\end{definition}

At its core, being a metric triple means that a metric triangle exists with lengths $a$, $b$, and $c$. In this regard, metric triples are foundational to categorizing metric spectra as metric triangles are the smallest nontrivial building blocks of metric spaces.

\begin{definition}
    We say that $S\subseteq \RR^+$ has the \textit{$4$-value condition} if for every pair of metric triples $(a,b,e), (c,d,e) \in S^3 $, there is an $f\in S$ such that $(b,c,f)$ and $(a,d,f)$ are metric triples.  
\end{definition}

\begin{figure}[htp]
    \centering
  \begin{tikzpicture}
      
         \node[fill,circle] (v_1) at (45 :2cm) {};
         \node[fill,circle] (v_2)  at (135 :2cm){};
         \node[fill,circle] (v_3)  at (225 :2cm){};
         \node[fill,circle] (v_4) at (315 :2cm){};

      \draw[black] (v_1) edge["c"] (v_2);
      \draw[black] (v_2) edge["b"] (v_3);
      \draw[black] (v_3) edge["a"] (v_4);
      \draw[black] (v_4) edge["d"] (v_1);

      \draw[black] (v_2) edge["e"] (v_4);
      \draw[red] (v_1) edge["f"] (v_3);
   
   \end{tikzpicture}
    \caption{Geometric interpretation of the $4$-value condition.}
    \label{fig:enter-label}
\end{figure}
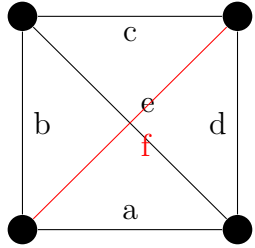

While not pertinent to our paper, an $S$ satisfying the $4$-value condition is equivalent to the associated class of finite metric spaces having the \textit{amalgamation property}. Hence, by better understanding the combinatorics of $S$, we better understand the category of finite metric spaces whose distances are contained in $S$. This leads to an intriguing abstractification whereby we study binary relational structures (metric spaces), by considering a ternary relational structures on sets of reals.

\begin{definition}
    Given $S,T\subseteq \RR^+$, we say $S\sim T$ ($S$ is \textit{equivalent} to $T$) if and only if there is an increasing bijection $f:S\rightarrow T $ with the following property:
    \begin{itemize}
        \item $(x,y,z)\in S^3$ is a metric triple if and only if $(f(x),f(y),f(z))\in T^3$ is a metric triple. 
    \end{itemize}
\end{definition}

An equivalence between sets of reals naturally gives rise to an isomorphism of category for the associated family of metric spaces. Thus, $\sim$ is a combinatorial relation that holds functorial information. \\
\\
The field of metric geometry is vast, and we have only covered a subset of a subset of the entire field. For brevity, we will end this section here as we already have more than enough definitions and context to proceed. However, we would be remiss to not at the very least offer readers other texts to travel further down the rabbit hole. For further readings on relating sets of reals and their combinatorial properties to categories of metric spaces, we recommend Ma\v{s}ulovi\'cs' work on big Ramsey degrees of universal structures \cite{Mas}. For more on structural properties of metric spaces, we recommend Ne\v{s}et\v{r}ils' paper \textit{Metric Spaces are Ramsey} \cite{metareramsey}. This paper also applies the partite construction, a technique that is foundational to structural Ramsey Theory.

\subsection{Introduction to Cones}

We found a way to apply linear optimization and polygonal geometry to better understand $\sim$-classes of finite subsets of $\RR^+$. This requires we vectorize sets $S\subseteq \RR^+$, which leads to our first definition.

\begin{definition}
A \textit{spectrum} is a vector $\vec{a}=(a_1,...,a_n) \in \RR^n$ with the property that $0< a_1<...<a_n $. A \textit{semispectrum} is a vector $\vec{a} \in \RR^n$ with the property that $0\leq a_1\leq...\leq a_n$. A spectrum is called \textit{integral} if it belongs to $\ZZ^n$.
\end{definition}

Spectra are then simply encodings for a finite set $S\subseteq \RR^+$, though viewing them instead as vectors is quite important to our analysis. It follows then that we need to transfer our definitions of the last subsecton into our new framework.

\begin{definition}
    We say two spectra $\vec{a}$ and $\vec{b}$ are equivalent, and write $\vec{a} \sim \vec{b}$ if for every (not necessarily distinct) triple $i,j,k \in \{1,...,n\}$, $a_i+a_j \geq a_k$ if and only if $b_i+b_j \geq b_k$. 
\end{definition}

%\begin{definition}
%    We say two semispectra $\vec{a}$ and $\vec{b}$ are equivalent, and write $\vec{a} \sim \vec{b}$ if
%    \begin{itemize}
%        \item For every (not necessarily distinct) triple $i,j,k \in \{1,...,n\}$, $a_i+a_j \geq a_k$ if and only if $b_i+b_j \geq b_k$. 
 %       \item For every $i<j$, $a_i =a_j$ if and only if $b_i = b_j$.
  %  \end{itemize}
%\end{definition}

Notice that the last condition is a rephrasing of preservation of metric triple. We state it in this new way for good reason, as a triple being a metric triple is now equivalent to satisfying a linear inequality. This leads us to polygons. We start with (arguably) the simplest polygon, the \textit{convex cone}.

\begin{definition}
    We say $C\subseteq \RR^n$ is a \textit{convex cone} if:
    \begin{itemize}
        \item For all $\alpha,\beta \geq 0$, $\vec{x},\vec{y} \in C$, $\alpha \vec{x} + \beta \vec{y} \in C$.
    \end{itemize}
    We say a cone $C$ is \textit{pointed} if $C\cap -C = \{0\}$.
\end{definition}

\begin{definition}
    Let $\vec{e}_i$ denote the $i$th canonical basis vector of $\RR^n$ (where $n$ will be clear from context). 
\end{definition}

\begin{lemma}
Let $A:\RR^m \rightarrow \RR^n$ be a linear transformation. The set of solutions to the equation $A\vec{x} \geq 0$ forms a closed pointed convex cone.
\end{lemma}

\begin{proof}
The set of vectors with positive entries is closed and $A$ is continuous, so it is clear the set of solutions forms a closed set. The convex cone conditions are immediate and easy to verify from the linearity of $A$. 
\end{proof}

We now define a family of vectors $\vec{e}_{(i,j,k)}$ that will essentially be used to bookkeep whether or not the triple $(x_i,x_j,x_k)$ is metric.

\begin{definition}\label{vectors}
    Let $(i,j,k) \in \{1,...,n\}^3 $ with $i\leq j < k $ and let $\vec{e}_i$ denote the $i$th canonical basis vector. We let $\vec{e}_{(i,j,k)}$ denote the vector in $\RR^n $ given by 

$$   
\vec{e}_{(i,j,k)} =\vec{e}_i+\vec{e}_j-\vec{e_k} 
%\begin{cases}
%0 & l\notin \{i,j,k\}\\
%2 & l=i=j\\
%1 &= l=i, \; i\neq j\\
%1 &= l=j, \; i\neq j\\
%-1 & l=k
%\end{cases}
$$    
\end{definition}

We now classify $\sim$-classes as a subset of a convex pointed cone. 

\begin{definition}
    Let $\vec{a} $ be a spectrum. Let $E$ be the family of vectors defined as follows.
    \begin{itemize}
        \item For all $(i,j,k) \in \{1,...,n\}^3$, $i\leq j <k$, either $e_{(i,j,k)}^T\in E $ or $-e_{(i,j,k)}^T \in E$.
        \item For all $(i,j,k) \in \{1,...,n\}^3$, $i\leq j <k$, if $a_i+a_j \geq a_k$, then $e_{(i,j,k)}^T \in E$.
        \item For all $i,j \in \{1,...,n\}$, $i<j$, $p_{(i,j)}$ which is the unique vector defined by $p_{(i,j)}(i)=-1$, $p_{(i,j)}(j) = 1$, and $0$ otherwise.
        \item For all $i \in \{1,...,n\}$, the transposed canonical basis vector, $e^T_i$.
    \end{itemize}
    Consider a matrix $A$ who's row vectors are the vectors in $E$. We call $A$ a matrix associated to $\vec{a}$.
\end{definition}

\begin{lemma}\label{simcones}
    The closure of a $\sim$-class is a pointed convex cone. 
\end{lemma}

\begin{proof}
    Fix a spectrum $\vec{a}$ and let $A$ be a matrix associated to $\vec{a}$. First notice that all solutions of $A\vec{x}\geq 0$ are necessarily semispectra. Also, if $\vec{b} \sim \vec{a}$, then $A \vec{b} \geq 0$. As a consequence, the set of solutions to the inequality $A\vec{x} \geq 0$ contains the $\sim$-class and is a closed pointed convex cone. Now if $\vec{b}$ is such that $A\vec{b} \geq 0$, but $\vec{b} \nsim \vec{a}$, it is precisely because for some $i,j,k$, $b_i+b_j = b_k$, but $ a_i+a_j<a_k$. Note that solutions to $A\vec{x} >0$ are exactly spectra $\sim$-equivalent to $\vec{a}$, and so we can find a sequence of members from the $\sim$-class of $\vec{a}$ that converges to $\vec{b}$ as $\{\vec{x}: A\vec{x} >0 \}$ is the connected interior of the cone $ \{\vec{x}: A\vec{x} \geq 0 \} $, and $\vec{b}$ is on the boundary. Hence, $\{\vec{x}: A\vec{x} \geq 0 \}$ is precisely the closure of the $\sim$-class of $\vec{a}$. 
\end{proof}

\begin{corollary}
Members of the closure of a $\sim$-class are semispectra. 
\end{corollary}

\begin{corollary}\label{scaleinvariant}
    If $\vec{x}$ is a spectra and $\alpha>0$, $\alpha \vec{x} \sim \vec{x}$.
\end{corollary}

The above corollary is already a well known fact that is easy to show. However, using that the closure of $\sim$-classes are cones provides an alternative geometric point of view. Surpisingly, a lot of insight into spectra can be deduced by using geometric reasoning. In the next section, we introduce \textit{polyhedra}. For all material on cones and polyhedra, we used \textit{Polyhedral Geometry and Linear Optimization} by Paffenholz as a reference \cite{Paffenholz}. We advise the curious reader to consult it as well for more on polyhedra and linear optimization. 
\section{Polyhedra and Bounds}

\subsection{Polyhedrons and Spectra}

As we've just seen, semispectra are classified completely by cones. However, our interests lie in integral spectra, so the class of semispectra is too broad. We mitigate this by instead considering \textit{polyhedrons}. Similar to the case of cones, polyhedrons naturally arise as the solution set to a family of linear inequalities. 

\begin{definition}
    A \textit{polyhedron} is a set of the form $\{x\in \RR^n : A\vec{x} \geq \vec{b} \}$ where $A:\RR^n \rightarrow \RR^m$ is a linear map.
\end{definition}

Phrased another way, a polyhedron is the solution set to a list of finitely many linear inequalities, and hence the intersection of finitely many half-spaces.

\begin{definition}
    Given a matrix $A$, we let $A_{i,*}$ denote the $i$th row. If $A$ has $m$-rows, and $I\subseteq [m]$, we let $A_{I,*}$ denote the submatrix consisting of the rows of index $I$. Similarly, in the case of a vector $\vec{b}$, we let $\vec{b}_I$ denote the subvector restricted to the indices in $I$. 
\end{definition}

\begin{definition}
    Given a polyhedron $P$ defined by the equation $A\vx \geq \vb$ where $A$ has $m$ rows, and $I\subseteq [m]$, a \textit{face} of a polyhedron is a set of the form $F= \{ \vec{x}\in P:  A_{I,*}\vec{x} = \vec{b}_I  \}$.
\end{definition}

If we interpret a polyhedron as an intersection of finitely many half-spaces, then faces are then the intersections of finitely many planes. The most significant faces are those that are \textit{minimal}. Especially in the case that the minimal face is zero dimensional.

\begin{definition}
    Given a polyhedron $P$ defined by the equation $A\vx \geq \vb$, we say that $\vec{x} \in P$ is a \textit{vertex} if provided the maximal $I$ such that $A_{I,*} \vec{x} = \vec{b}_I$, $A_{I,*}$ has rank $n$. If a vertex exists, we say $P$ is \textit{pointed}.
\end{definition}

\begin{theorem}[Folklore \cite{Paffenholz}]\label{polyfund}
    Given a nonempty polyhedron $P$ defined by the equation $A\vx \geq \vb$, the following are equivalent:
    \begin{itemize}
        \item $A$ has rank $n$.
        \item There is a vertex $\vec{v} \in P$.
    \end{itemize}
\end{theorem}

The above is a fundamental result of polyhedrons that tightly correlates geometry to linear optimization. Moreover, it says something rather significant. Namely, a nonempty intersection of a half space and a pointed polyhedron is a pointed polyhedron. 

\begin{definition}
    Let $\vec{a}$ be a metric spectrum. The \textit{spectral polyhedron corresponding to } $\vec{a}$, denoted $P(\vec{a})$, is the polyhedron defined by the equation $A\vec{x} \geq \vec{b}$ where
    \begin{align*}
        b_i &= \begin{cases}
                1 & \textnormal{if } A_{i,*} = -e_{j,k,l}^T  \textnormal{ or } e^T_k \textnormal{ or } p_{i,j} \textnormal{ where } i<j\\
                0 & \textnormal{otherwise}
                \end{cases}
    \end{align*}
    Hence, if $a_i+a_j \leq a_k$ and $\vec{x} \in P(\vec{a})$, then  $x_i+x_j \leq x_k-1  <x_k$, and the sequence is strictly increasing and starting in $x_1\geq1>0$.  
\end{definition}
We want strict inequalities to avoid having to work with semispectra. As it is written now, spectral polyhedrons have two unique properties. The first, is that they only contain spectra, not semispectra, by the first two lines in the previous definition. Next, they are also subsets of the associate $\sim$-class as we will see in Lemma \ref{simpoly}.
\begin{lemma}
    If $\vec{x}$ is a spectrum, then $P(\vec{x})$ is a non-empty polyhedron.
\end{lemma}
\begin{proof}
    We will prove that there is an $\alpha>0$ such that $\alpha \vec{x} \in P(\vec{x})$. This is equivalent to $\alpha(A\vec{x})=A\alpha \vec{x}\geq \vec{b}$. Now we observe $\vec{s}=A\vec{x}$. Since $\vec{x}$ it is easy to show that if $b_i>0$, then $s_i>0$ and since $\vec{s}=A\vec{x}\geq 0$ ,because every spectrum satisfies its own defining inequalities, we can always scale $\vec{s}$ by $1$ over its smallest non-zero component to get a vector larger than $\vec{b}$ which finishes the proof.
\end{proof}
\begin{lemma}\label{simpoly}
    Let $\vec{x}$ and $\vec{y}$ be metric specrta. $\vec{x} \sim \vec{y} $ if and only if $P(\vec{x}) = P(\vec{y})$.
\end{lemma}

\begin{proof}
    First, note that if $\vec{x} \sim \vec{y} $, then the set of defining vectors $E$ from definition \ref{vectors} is identical for both $\vec{x}$ and $ \vec{y}$. Consequently,  $P(\vec{x})$ is a subset of the $\sim$ class of $\vec{x}$, and since $P(\vec{x})$ is non-empty, for the other direction take any $t\in P(\vec{x})=P(\vec{y})$, then $x\sim t\sim y$, so the result is immediate. 
\end{proof}

Spectral polyhedrons are significant in that they are pointed.

\begin{prop}\label{existsvertex}
    Given a nonempty spectral polyhedron $P = \{\vx \in \RR^n : A\vx \geq \vb \} $, there is a vertex $\vec{v}\in P$. 
\end{prop}
\begin{proof}
    By Theorem \ref{polyfund}, it suffices to show that $A$ has rank $n$. Recall that the matrix $A$ must contain the vectors of the form $\vec{e}_i$ as rows. This family clearly has full rank and spans all of $\RR^n$ itself, hence regardless of the other rows in $A$, $A$ must have full rank.
\end{proof}

Notice the power above of Theorem \ref{polyfund}, we are able to deduce that the polyhedron is pointed as it is the intersection of a pointed polyhedron and half-spaces. This provides a fairly useful existence result. As a corollary to the above, we will show that every metric spectra admits a rational (hence integral) representative. This was shown by Conant using quantifier elimination \cite{ConantThesis}. We instead prove this using polyhedral geometry. The advantage of our approach is that it is of a computational nature, and hence provides a concrete representative opposed to only proving existence. Our ultimate result will share a lot in common with the below proof.
\begin{theorem}{(Conant \cite{ConantThesis})}
    Given a metric spectra $\vx \in \RR^n$, there is a spectra $\vec{y} \in \QQ^n $ such that $\vec{y}\sim \vx$.
\end{theorem}
\begin{proof}
    Let $P$ denote the spectral polyhedron associated to $\vx$. Let $\vec{y}$ denote a vertex, which we know must exist by Proposition \ref{existsvertex}. We know $\vec{y}\sim \vx$ by Lemma \ref{simpoly}. Let $B = A_{I,*}$ where $|I| = n$, $A_{I,*} \vec{y} = \vec{b}_I$, and $A_{I,*} $ has rank $n$. By Cramer's rule, we have an exact computation for $\vec{y}$ by the formula
    \begin{align*}
    y_i &= \frac{\textnormal{det}(B_i) }{  \textnormal{det}(B) }    
    \end{align*}
    where $B_i$ is achieved by replacing column $i$ in $B$ with $\vec{b}_I$. Note that $B$ and $B_i$ (for all $i$) consist of integer entries. Hence, $y_i\in \QQ$. Thus, $\vec{y}\in \QQ^n$ which completes our proof. 
\end{proof}

\subsection{Covers and a Decomposition Lemma}

In order to later achieve bounds, we will need to better understand the combinatorics of matrices associated to metric spectra. We do this by constructing a partial order on $\RR^n$ that will help us develop bounds on determinants. Utilizing Cramer's rule as done previously in our proof of Conant's Theorem, we will compute well bounded integral representatives of $\sim$-classes.

\begin{definition}
    Given two vectors $\vec{x}$ and $\vec{y}$, we say $\vec{x} $ \textit{covers} $\vec{y} $, written $\vec{x} \preceq \vec{y}$, if for all $i$, $x_i y_i \geq 0 $ and $|x_i| \leq |y_i|$. We also say $S\subseteq \RR^n$ \textit{covers} $ \vec{x}$, $\vec{x} \preceq S $, if there is a $ \vec{y}\in S$ that covers $\vec{x}$.
\end{definition}

It is easy to verify that $\preceq$ is a partial order. Given the covering relation is determinant on sign of entries, we introduce the sign function.

\begin{definition}
    Given a nonzero $x\in \RR$, we say $\sgn(x) = 1$ if $x>0$, and $\sgn(x) = -1 $ otherwise. 
\end{definition}

We now introduce a covering decomposition theorem.

\begin{theorem}\label{decomp1}
    If $\vec{x}\preceq\vec{u}+\vec{v}$, there exist $\vec{x_u}\preceq\vec{u}$ and $\vec{x_v}\preceq\vec{v}$ such that $\vec{x}_u+\vec{x}_v=\vec{x}$.
\end{theorem}

\begin{proof}

      Let 

    \begin{align*}
        \vec{x}_u(i)=\begin{cases}
        0 & u_i\cdot v_i<0 \textnormal{ and } u_i\cdot x_i\leq0 \\
        x_i & u_i\cdot v_i<0 \textnormal{ and } u_i\cdot x_i> 0 \\
         \sgn(u_i)\cdot \textnormal{min}(|u_i|,|x_i|)   & \textnormal{otherwise}
    \end{cases}
    \end{align*}

Trivially we have $|\vec{x}_u(i)|\leq |u_i|$ in the first and last case. To see this is also true in the second case, suppose $u_i \cdot v_i < 0$ and $u_i \cdot x_i >0$. Since $\sgn(x_i) = \sgn(u_i)$ and $\sgn(u_i) \neq \sgn(v_i)$, $x_i \cdot v_i <0$. Since $\vec{x} \preceq \vec{u} + \vec{v} $, we acquire the following bounds 
\begin{align*}
    0 &\leq x_i\cdot (u_i+v_i) \\
    - x_i \cdot u_i&\leq x_i\cdot v_i <0 \\
    &\Rightarrow\\
    -1 &\geq -\frac{v_i}{u_i} >0\\
    &\Rightarrow \\
    \frac{|v_i|}{|u_i|} &\leq 1\\
    |v_i| &\leq |u_i|
\end{align*}
We also know that $|x_i| \leq |u_i +v_i|$. However, $u_i$ and $v_i$ have different signs, with $|u_i| \geq |v_i|$. Hence, $|x_i| \leq |u_i+v_i| \leq |u_i|$, and so $|\vec{x}_i(i)|\leq |u_i|$ is always true. It is even easier to see that $\vec{x}_u(i))u_i \geq 0$ is true for every $i$. Hence, $\vec{x}_u \preceq \vec{u}$.\\
\\
We now only have one choice for $\vec{x}_v$, namely $\vec{x} - \vec{x}_u$. We show now that $\vec{x}_v \preceq \vec{v}$. The first case is easy to verify. In this case, $\vec{x}_v(i) = x_i$, and $x_i$ has the same sign as $v_i$. Since $v_i$ has a different sign from $u_i$, similar to the previous argument, we have $|x_i|  \leq |v_i|$. The second case is even easier to verify as $\vec{x}_v(i) = 0$ in this case. We now check the more complicate case, the final one.\\
\\
In this case, $u_i$ and $v_i$ have the same sign. By our assumption, $x_i$ has the same sign as $u_i + v_i$. Thus, 
\begin{align*}
    |x_i -  \sgn(u_i)\cdot \textnormal{min}(|u_i|,|x_i|)| &= | |x_i| - |\textnormal{min}(|u_i|,|x_i|) | |
\end{align*}
If $|x_i| \leq |u_i|$, the above value is zero and so $|\vec{x}_v(i)| \leq |v_i| $ and has the same sign. In the case that $|x_i| > |u_i|$, the above quantity is $|x_i| - |u_i|$. But then, by triangle inequality
\begin{align*}
    |x_i| &\leq |u_i+v_i|\\
    &\leq |u_i| + |v_i|\\
    &\Rightarrow \\
    |x_i| - |u_i| &\leq |v_i|
\end{align*}
Hence, $\vec{x}_v \preceq \vec{v}$ as desired.

\end{proof}

 As it is not the case that the $\preceq$ order is algebraic (that is to say $\vec{x} \preceq $
%\preceq
$\vec{y}$ does not guarantee $\vec{x} + \vec{z} \preceq \vec{y} + \vec{z} $), we want a way to refine vectors to achieve some type of algebraicity in the $\preceq$ order. We show that this can be achieved when we restrict our attention to a simple (but relevant) family of vectors. 

\begin{definition}
    Let $\mathcal{F} = \{\vec{p}_{(i,j)}  : i,j \in \{1,...,n\} \; i\neq j \}\cup \{\vec{e}_i : i \in \{1,...,n\} \} \cup \{\vec{e}_i : i \in \{1,...,n\} \}  $ where $p_{(i,j)}$ are as they are defined in Definition \ref{vectors}.
\end{definition}

\begin{prop}\label{decomp2}
    Suppose $\vec{x}$ and $\vec{y}$ are such that $\vec{x} \preceq \vec{p}_1$ and $\vec{y} \preceq \vec{p}_2$ where $\vec{p}_1, \vec{p}_2 \in \mathcal{F} $. There are refinements $\vec{q}_1 \preceq \vec{p}_1$ and $\vec{q}_2 \preceq \vec{p}_2$ such that $\vec{x} + \vec{y} \preceq \vec{q}_1 + \vec{q}_2$ and $\vec{q}_1,\vec{q}_2 \in \mathcal{F} $. 
\end{prop}

\begin{proof}
Let $\vec{z} = \vec{x}+ \vec{y} $. We define $\vec{q}_1$ and $\vec{q}_2$ coordinatewise and simultaneously with the following equality:

\begin{align*}
    (\vec{q}_1(i),\vec{q}_2(i))=\begin{cases}
        (\vec{p}_1(i),\vec{p}_2(i))& \vec{p}_1(i)\vec{p}_2(i)\geq 0 \\
        (\vec{p}_1(i),0) & \vec{p}_1(i)\vec{z}(i)\geq 0\textnormal{ and } \vec{p}_1(i)\vec{p}_2(i)< 0\\
        (0,\vec{p}_2(i))&  \vec{p}_2(i)\vec{z}(i)\geq 0\textnormal{ and } \vec{p}_1(i) \vec{p}_2(i)<0\textnormal{ and } \vec{p}_1(i)\vec{z}(i)< 0
    \end{cases}
\end{align*}

We claim that this is a well defined notion, meaning the above three cases are exhaustive. To see this, fix some $i$. Suppose the first two conditions fails. Since $\vec{p}_1(i) \vec{p}_2(i)  <0 $ from the first condition, the second condition failing must mean that $\vec{p}_1(i) \vec{z}(i) <0 $. It suffices now to verify that $\vec{p}_2(i) \vec{z}(i) \geq 0$ to deduce the third condition. However, note that we have $\sgn( \vec{p}_2(i)) \neq \sgn(\vec{p}_1(i)) \neq \sgn(\vec{z}(i))$. Hence, $\sgn(\vec{p}_2(i)) = \sgn(\vec{z}(i))$, and so we are done. \\
\\
Given $\vec{q}_1(i)$ (resp. $\vec{q}_2(i)$) is either $\vec{p}_1(i)$ (resp. $\vec{p}_2(i)$) or $0$, we have that $\vec{q}_1\preceq \vec{p}_1$ and $\vec{q}_2 \preceq \vec{p}_2$. For the more challenging part of the proof, we must verify that $\vec{z} \preceq \vec{q}_1 + \vec{q}_2$. We do this again by cases. In the first case, $ \vec{p}_1(i)\vec{p}_2(i)\geq 0 $. Since $\vec{x} \preceq \vec{p}_1$ and $\vec{y}  \preceq \vec{p}_2$, $|\vec{x}(i)| \leq |\vec{p}_1(i)|$, and $|\vec{y}(i)| \leq |\vec{p}_2(i)| $. Since $\vec{p}_1(i)$ shares the same sign as $\vec{p}_2(i) $,
\begin{align*}
    |\vec{x}(i) + \vec{y}(i)| &\leq |\vec{p}_1(i) + \vec{p}_2(i)|\\
    &= |\vec{p}_1(i)|+ | \vec{p}_2(i)|\\
    &=|\vec{q}_1(i)|+ | \vec{q}_2(i)|
\end{align*}
We also have that
\begin{align*}
    \sgn(\vec{x}(i)) &= \sgn(\vec{p}_1(i)) \\
    &= \sgn(\vec{p}_1(i)+\vec{p}_2(i))\\
    &= \sgn(\vec{y}(i))\\
    &\Rightarrow\\
    \sgn(\vec{z}(i)) &= \sgn(\vec{q}_1(i)) + \sgn({q}_2(i))
\end{align*}
We now go to the second case. In this case, we have 
\begin{align*}
    |\vec{z}(i)| &= \sgn(\vec{z}(i)) \vec{z}(i)\\
    &= \sgn(\vec{p}_1(i)) \vec{z}(i)\\
    &= \sgn(\vec{p}_1(i)) \vec{p}_1(i) + \sgn(\vec{p}_1(i)) \vec{p}_2(i)\\
    &= | \vec{p}_1(i)| +\sgn(\vec{p}_1(i)) \vec{p}_2(i)
\end{align*}
By our hypotheses, $\vec{p}_1(i)$ and $ \vec{p}_2(i)$ have different signs. Hence,

\begin{align*} 
     |\vec{z}(i)| &\leq | \vec{p}_1(i)| = |\vec{q}_1(i)+\vec{q}_2(i)|
\end{align*}
We also have that $\vec{q}_1(i)+\vec{q}_2(i)$ has the same sign as $ \vec{p}_1(i)$, and hence the same sign as $\vec{z}(i)$. This completes the second case. The third case is near identical to the second case, and so we will skip its proof out of the interest of brevity. It should also be immediate from how we defined $\vec{q}_1$ and $\vec{q}_2$ that in every case, $\vec{q}_1,\vec{q}_2 \in \mathcal{F}$.
\end{proof}

We conclude this section with a corollary that will be foundational to our computations. 

\begin{corollary}
    If $ \vec{x} \preceq \mathcal{F} $ and $\vec{x} \preceq \mathcal{F} $, then $\vec{x}+\vec{y} \preceq \mathcal{F} + \mathcal{F} $. 
\end{corollary}

\begin{proof}
    Take $\vec{p}_1,\vec{p}_2 \in \mathcal{F} $ such that $\vec{x}\preceq \vec{p}_1$, and $\vec{y}\preceq \vec{p}_2 $. Let $\vec{z}=\vec{x}+\vec{y}$. By the construction in Theorem \ref{decomp2}, we have a pair $\vec{q}_1 \preceq \vec{p}_1 $ and $\vec{q}_2 \preceq \vec{p}_2$ such that $\vec{x}+\vec{y} \preceq \vec{q}_1 + \vec{q}_2$.  \\
    \\
    \textbf{Case 1:} Suppose that for every $i$, $\sgn(\vec{p}_1(i)) = \sgn(\vec{p}_2(i))$. By our definitions for $\vec{q}_1$ and $\vec{q}_2$, $\vec{q}_1=\vec{p}_1$ and $ \vec{q}_2 = \vec{p}_2$. Hence, we are done.\\
    \\
    \textbf{Case 2:} There is an $i$ such that $\sgn(\vec{p}_1(i)) \neq \sgn(\vec{p}_2(i))$. The only way this can happen is if one of the values is $1$, and the other is $-1$. Moreover, since each vector contains exactly one $1$ and one $-1$, there is a unique $i$ for which this occurs. Without loss of generality, suppose $\vec{p}_1(i) = -1$ and $\vec{p}_2(i) = 1$. \\
    \\
    \textbf{Subcase 1:} If $\vec{p}_1(i) \vec{z}(i) \geq 0$, then $\vec{z}(i) \leq 0$. It follows that $\vec{q}_1(i) + \vec{q}_2(i) = -1$. We now know that $\vec{q}_1+\vec{q}_2$ is a vector with a positive $1$ followed by two $-1$'s, and zero elsewhere eg. $[...1...-1...-1...] $. It is clear that this can be expressed as the sum of two vectors in $ \mathcal{F}$, namely $\vec{p}_{(i,j)} - \vec{e}_k $ where $i<j<k$ are the locations of the nonzero entries. \\
    \\
    \textbf{Subcase 2:} If $\vec{p}_2(i) \vec{z}(i) \geq 0$, then $\vec{z}(i) \leq 0$. It follows that $\vec{q}_1(i) + \vec{q}_2(i) = 1$. Similar to the previous subcase, we now know that $\vec{q}_1 + \vec{q}_2$ is a vector with two positive $1$'s followed by a $-1$, and zero elsewhere eg. $[...,1,...,1,...,-1,...] $. This can be expressed as the sum of two vectors in $\mathcal{F}$, namely $\vec{p}_{(i,k)} + \vec{e}_j$ where $i<j<k$ are the locations of the nonzero entries.
\end{proof}

\subsection{Bounds}

We now get in to computing upper bounds. As we've alluded to up to this point, representatives will be vertices of a polygon that are computed using Cramer's rule. To achieve bounds, we inductively show our matrices will be reasonably covered by $\mathcal{F}+\mathcal{F}$. Given we would like to implement row operations on matrices to simplify calculations, we prove the following lemma.

\begin{lemma}\label{rowoplemma}
Let $\vec{x},\vec{y}\preceq \mathcal{F}$ and $j$ minimal such that $0\neq|x_j|= ||\vec{x}||_\infty $. Then $\vec{y}-\frac{y_j}{x_j}\vec{x}\preceq \mathcal{F}$.    
\end{lemma}
\begin{proof}
    Since $\vec{x}$ and $ \vec{y}$ are covered by $\mathcal{F}$, both vectors contain at most two nonzero entries, and at most one entry of each sign. \\
\\
\textbf{Case 1:} The vector $\vec{x}$ only contains one nonzero value. In this case, it is precisely $x_j$. It follows that $\vec{y} - \frac{y_j}{x_j}\vec{x} $ is simply $\vec{y}$ with the $j$th coordinate turned to $0$. It follows then that $\vec{y} - \frac{y_j}{x_j}\vec{x}$ has at most one nonzero term of cardinality bounded by $1$, and hence is still covered by $\mathcal{F}$.
\\
\\
\textbf{Case 2:} The vector $ \vec{x}$ contains two nonzero values. Let $k$ be the index of the nonzero value that isn't $j$. 
\begin{align*}
    (\vec{y} - \frac{y_j}{x_j}\vec{x})(i) &= 
    \begin{cases}
        0 & i=j \\
        y_k - \frac{x_k}{x_j} y_j & i=k \\
        y_i & \textnormal{otherwise.}
    \end{cases}
\end{align*}
It follows that $\vec{y} - \frac{y_j}{x_j}\vec{x}$ still has at most two nonzero coordinates. It will have precisely two nonzero coordinates in the case that $y_k=0$.\\
\\
\textbf{Subcase 1:} Suppose $y_k=0$. It suffices to show that the $k$th coordinate has the same sign as the coordinate that had been turned to $0$ ($j$) and has cardinality bounded by $1$. First, we check the sign. Recall that $\sgn(x_j) \neq \sgn(x_k)$. So, $\sgn( - \frac{x_k}{x_j} y_j) = \sgn(y_j)$. Given $|\frac{x_k}{x_j}| \leq 1$ and $|y_j| \leq 1$, $|- \frac{x_k}{x_j} y_j| \leq 1$ as desired. 
\\
\\
\textbf{Subcase 2:} Suppose $y_k \neq 0$. It follows that $ \vec{y} - \frac{y_j}{x_j}\vec{x}$ only has one nonzero entry. It suffices then to show it is bounded in absolute value by $1$. Note that $\sgn(y_k) \neq \sgn(y_j)$. Moreover, as we saw in the last subcase, $\sgn( - \frac{x_k}{x_j} y_j) = \sgn(y_j)$. It follows then that 
\begin{align*}
     |y_k - \frac{y_j}{x_j}x_k| &= \textnormal{max} \{y_k - \frac{y_j}{x_j}x_k, \frac{y_j}{x_j}x_k -y_k  \}\\
     &\leq \textnormal{max}\{y_k, \frac{y_j}{x_j}x_k  \}\\
     &\leq 1
\end{align*}
This completes this subcase, and thus our proof. 
\end{proof}

\begin{theorem}\label{detbound}
       If $A$ is an $n\times n$ matrix with the property that every row is covered by $\mathcal{F} + \mathcal{F}$, then $|\Det A|\leq 2^{n}$. Moreover, whenever an $n\times n$ matrix $A$ has $m\geq0$ rows covered by $\mathcal{F}$, then $|\Det A|\leq 2^{n-m}$.
\end{theorem}

\begin{proof}
        We proceed by induction on $n$. It is clear that the statement is true for $n=1$. Suppose now that the statement is true for any $m>0$ and $n^\prime <n$. We prove the statement is true for $n$. Let $A$ be an $n$ by $n$ matrix with rows $\vec{r}_i\preceq \mathcal{F}+ \mathcal{F}$, and $m$ rows covered by $\mathcal{F}$ for some $m$. There is nothing to verify in the case that a row is the zero vector, so suppose every row is nonzero. \\
        \\
        Using the first row $\vec{r}_1$, apply to each row the operation from Lemma \ref{rowoplemma} to return a new matrix $A^\prime$. Since we performed scale row operations, we have $\Det A = \Det A^\prime $. We also have by Lemma \ref{rowoplemma} that every row of $A^\prime$ is covered by $\mathcal{F}+\mathcal{F}$. Let $j$ be the smallest integer such that $\vec{r}_1(j) = ||\vec{r}_1||_\infty $. By our matrix operation, we have that the $j$th column is only nonzero in the first entry. Hence, to compute the determinant of $A^\prime$, we apply the Laplace expansion along the $j$th column. It follows that $\Det A^\prime = ||\vec{r}_1||_\infty A^{\prime \prime} $, where $A^{\prime \prime}$ is $A$ with the first row and $j$th column removed. This new matrix minor is $n-1$ by $n-1$ and satisfies that each row is covered by $\mathcal{F}+\mathcal{F}$. We now split into cases.\\
        \\
        \textbf{Case 1:} $\vec{r}_1 \preceq \mathcal{F}$. It follows then that $||\vec{r}_1|| \leq 1$, and $A^\prime$ has at most $m-1$ rows covered by $\mathcal{F} $. By our induction hypothesis, $|\Det A^{\prime \prime}| \leq 2^{n-1 - (m-1)} = 2^{n-m}$. Hence, $|\Det A| \leq ||\vec{r}_1||_\infty \cdot 2^{n-m} \leq 2^{n-m} $.\\
        \\
        \textbf{Case 2:} $\vec{r}_1 $ cannot be covered by $\mathcal{F}$. It follows that $A^{\prime \prime}$ has $m$ rows that can be covered by $\mathcal{F}$, so $|\Det A^{\prime \prime}| \leq 2^{n-1 - m}$. A simple application of the triangle inequality yields that $||\vec{r}_1 ||_\infty \leq 2$. Hence, $|\Det A| \leq 2 \cdot 2^{n-1-m} = 2^{n-m}$.\\
        \\
        This completes our proof. 
\end{proof}

We can now prove the first part of our main result, showing that every spectra admits an integral $\sim$-class representative that is bounded by $2^n$.

\begin{theorem}\label{upperbound}
    Given a metric spectra $\vx$, there is an integral spectrum $\vec{y} \sim \vx$ such that $y_n \leq 2^n$. 
\end{theorem}

\begin{proof}
    Let $P$ be the polyhedron associated to $\vx$, and let $A\vx \leq \vec{b}$. By Theorem \ref{existsvertex}, there is a vertex $\vec{y}\sim \vx $ that is a vertex of $P$. It follow that there is an associated $n$ by $n$ matrix $B= A_{I,*}$ such that $B\vec{y} = \vec{b}_I$. Let $B_n$ denote the matrix achieved by replacing the $n$th column of $B$ with $-\vec{b}_I$. By Cramer's rule,
    \begin{align*}
        y_n & = \frac{-\det(B_n)}{\det(B)}\\
        &=  \frac{| \det(B_n)|}{|\det(B)|}
    \end{align*}
    where the last equality is true as metric spectra are positive. Note that we have, for practical reasons, negated the last column of the standard denominator determinant in Cramer's rule, hence the strange -1. To bound $y_n$, it suffices to find a reasonable upper bound to $B_n$ since $|\det(B)| \geq 1$ is an integer.\\
    \\
    \textbf{Claim:} Every row of $B_n$ is covered by $\mathcal{F}+\mathcal{F}$. This can be seen with simple case work.
    \\
    \\
    \textbf{Case 1:} The row $ \vec{r}$ of $B_n$ ends in a $0$. In this case, the row in $B$ was not of the form $e_{(i,j,k)}$ or $p_{(i,j)} $. It follows that the row is a vector of the form $-e_{(i,j,k)} $ with the last entry turned to $0$. Hence, it is covered by  $\mathcal{F}+\mathcal{F}$, as $e_{(i,j,k)} $ is covered by $\vec{p}_{(i,k)} + \vec{p}_{(j,n)} $. \\
    \\
    \textbf{Case 2:} The row $ \vec{r}$ of $B_n$ ends in a $-1$. This means the row must be of the form $e_{i,j,k}$ , $p_{(i,j)} $ or $e_i$. Similar to the previous case, it will be covered by $\vec{p}_{(i,k)} + \vec{p}_{(j,n)} $.
    \\
    \\
    Hence, by Theorem \ref{detbound}, $y_n |\det(B)| \leq 2^n$. Since for every $i$, $y_i = \frac{|\det(B_i)|}{|\det(B)|}$, $|\det(B)| \vec{y} \sim \vec{y}$ by \ref{scaleinvariant} and also has that its $n$th entry is bounded by $2^n$.
\end{proof}

Recall that vertices are, in some sense, extremal. They are solutions to a linear minimization problem after all. As a consequence, vertices in a spectral polygon also have \textit{internal} bounds that can be exploited to achieve lower bounds for a representative, at the small cost of increasing the upper bound.

\begin{theorem}
    If $\vec{v}$ is a vertex of $P(v)$, $2^{n-i+1}v_i\geq v_n$ , for all $i<n$.\label{lowerbound}
\end{theorem}
\begin{proof}
    Since $v$ is a vertex, there is an associated $n$ times $n$ matrix $B=A_{I,*}$ such that $B\vec{v}=\vec{b}_I$ and $B$ has rank $n$.
    For a fixed $i$, let $S=\{\vec{s}_j:j<i\}$ such that $\vec{s}_j=\vec{e}_j-\frac{v_j}{v_i}\vec{e}_i$, so there are $i-1$ vectors in $S$. Note that $\vec{v}$ is orthogonal to every vector in $S$ by construction and vectors in $S$ are covered by $\mathcal{F}$ because $v_j<v_i$ for $j<i$. Vectors in $S$ are also linearly independent because projection of $S$ to first $i-1$ coordinates makes the standard basis of $\mathbb{R}^{i-1}$ since $F(\vec{s}_j)= F(\vec{e}_j)-F(\frac{v_j}{v_i}\vec{e}_i)=\vec{e}_j-\vec{0}$, where $F$ is the projection map. Note that $\vec{e}_j$ in the last equality means $j$th unit vector in $\mathbb{R}^{i-1}$ so it is a slight abuse of notation. Now we use a well-known linear algebra fact that every linearly independent set of vectors can be extended to a basis using vectors from a full rank set. Let $I'$ be the set of rows in $B$ such that they extend $S$ to a basis. Let $B'=[S^T,(B_{I',*})^T]^T$ and let $\vec{b'}=[\vec{0}_{i-1}^T,b_{I'}^{T}]^{T}$. Since rows of $B'$ form a basis, $\vec{v}$ is the unique solution to the equation $B'\vec{x}=\vec{b'}$. Moreover, since now we know rows of $B_n$ are covered by $\mathcal{F}+\mathcal{F}$ and all $i-1$ vectors from $S$ are covered by $\mathcal{F}$ and remain unchanged under replacing the last column by zeros since $j<i<n$, by Theorem \ref{detbound} we know $|\det(B'_n)|\leq 2^{n-i+1}$. Now by Crammer's rule $$\frac{v_i}{v_n}=\frac{\frac{|\det(B'_i)|}{|\det(B')|}}{\frac{|\det(B'_n)|}{|\det(B')|}}=\frac{|\det(B'_i)|}{|\det(B'_n)|}\geq \frac{|\det(B'_i)|}{2^{n-i+1}}$$It suffices to prove $|\det(B'_i)|\geq 1$, which we will do by proving $\det(B'_i)$ is a non-zero integer. Since $\vec{v}$ is a spectrum, $v_i\neq 0$, and so $\det(B'_i)\neq0$. Finally, by construction all entries of $B'_i$ are integers: all rows from $A$ are integral and the only entries from $S$ which are not are by construction in the $i$th column of $B'$, so they get replaced by zeros in $B'_i$, which finishes the proof.
\end{proof}
\begin{corollary}\label{bounded}
    Every spectrum is equivalent to an integral spectrum $\vec{y}$ such that $y_n\leq 2^n$ and $2^{n-i+1}y_i\geq y_n$.
\end{corollary}
\begin{proof}
    Follows directly from Theorem \ref{lowerbound} and the fact that construction in Theorem \ref{upperbound} is of the form $\vec{y}=c\cdot \vec{v}$ where $\vec{v}$ is a vertex.
\end{proof}
We can now prove a weaker solution to Conants' problem. 

\begin{theorem}
    Given a metric spectra $\vec{x} \in \RR^n$, there is a metric spectra $\vec{y} \sim \vec{x} $ such that for all $i$, $ 2^i \leq y_i \leq 2^{n+1} $.
\end{theorem}

\begin{proof}
    By Corollary \ref{bounded}, there is a metric spectra $\vec{z} \sim \vx$ such that $z_n \leq 2^n$ and $2^{n-i+1}z_i\geq z_n$. Let $c= \frac{2^{n+1}}{z_n} \geq 2$. Consider the vector $\vec{y}$ defined coordinate-wise by $y_i = \lceil c z_i \rceil$. It is easy to verify that $\vec{y}\sim \vec{z}$ utilizing that $c\geq 2$. We also have $y_n = 2^{n+1}$. It suffices to show that $y_i \geq 2^i$. 
    \begin{align*}
        y_i &= \lceil c z_i \rceil = \lceil 2^{n+1}\frac{z_i}{z_n} \rceil\geq \lceil \frac{2^{n+1}}{2^{n-i+1}} \rceil=2^i.
    \end{align*}.

\end{proof}

While this concludes our main result, much work still remains to be done. It is still not clear if Conant's question has a positive answer, though all signs currently point to yes. We believe we have provided a very strong first step towards solving the problem. Moreover, our results are conducive to computational verification for low dimensions. Given Conant's statement is true for $n\leq 6$ (which was proven via computer programming) we applied a brute-force check for $n = 7$. In fact, we currently have verified that Conant's statement is true for $n=7$. We would love to collaborate with those who believe they can find a faster algorithm and check for much larger $n$.

\section*{Acknowledgements}

The author would like to thank his instructor Dragan Ma\v{s}ulovi\'c for introducing him to this question. The author would also like to graciously thank Keegan Dasilva Barbosa for initially providing the question to Dragan, and for supervising the author during the entirety of the project. Tangentially, the author would like to thank The Fields Institute for Research in Mathematical Sciences, as it was thanks to their grant that Keegan was able to work in a supervisory capacity.

\end{document}